\documentclass{amsart}
\usepackage{hyperref}
\usepackage{amsthm,mathtools}
\usepackage[all]{xy}

\usepackage{mathrsfs}

\theoremstyle{plain}
\newtheorem{theorem}{Theorem}[section]
\newtheorem{conjecture}[theorem]{Conjecture}
\newtheorem{corollary}[theorem]{Corollary}
\newtheorem{definition}[theorem]{Definition}
\newtheorem{example}[theorem]{Example}
\newtheorem{lemma}[theorem]{Lemma}

\newtheorem{proposition}[theorem]{Proposition}
\newtheorem{remark}[theorem]{Remark}

\usepackage[dvipsnames]{xcolor}

\usepackage{amssymb,amsmath,tabularx,graphicx,float,placeins}

\usepackage{tikz}
\usetikzlibrary[calc,intersections,through,backgrounds,arrows,decorations.pathmorphing]

\newcommand{\bbE}{\mathbb{E}}
\newcommand{\bbP}{\mathbb{P}}
\newcommand{\cB}{\mathcal{B}}
\newcommand{\cD}{\mathcal{D}}
\newcommand{\fc}{\mathfrak{c}}
\newcommand{\fm}{\mathfrak{m}}

\newcommand{\Dyck}{\operatorname{Dyck}}
\newcommand{\Func}{\operatorname{Func}}
\newcommand{\loft}{\operatorname{loft}}

\newcommand{\mr}{\operatorname{mr}}
\newcommand{\Par}{\operatorname{Par}}
\newcommand{\Path}{\operatorname{Path}}
\newcommand{\RP}{\operatorname{RP}}
\newcommand{\Unif}{\operatorname{Unif}}
\newcommand{\Var}{\operatorname{Var}}
\newcommand{\WT}{\operatorname{WT}}
\newcommand{\wt}{\operatorname{wt}}

\definecolor{dark-violet}{RGB}{148, 0, 211}

\begin{document}

\title{Large-scale Rook Placements}
\author{Pakawut Jiradilok}
\date{\today}

\address{Department of Mathematics, Massachusetts Institute of Technology, Cambridge, MA 02139}
\email[P.~Jiradilok]{pakawut@mit.edu}

\begin{abstract}
For each certain ``nice" piecewise linear function $f:[0,1] \to [0,1]$, we consider a family of growing Young diagrams $\{\lambda(f,N)\}_{N=1}^{\infty}$ by enlarging the region under the graph of $f$. We compute asymptotic formulas for the number of rook placements of the shape $\lambda(f,N)$. We prove that the normalized cumulative X-ray of a uniformly random permutation, as the size of the permutation grows, exhibits a limit shape phenomenon.
\end{abstract}


\maketitle

\section{Introduction}
While rook placements are classical objects in combinatorics (cf. e.g. \cite[Chapter~7]{Rio02} and \cite[Chapter~2]{EC1}), there are many recent works in the literature studying them (e.g. \cite{BR06, BLRS14, BLRS16, Bar21}). Enumerative combinatorics of rook placements deals with problems of counting the number of ways to place a certain number of non-attacking rooks on a certain subset of a chessboard. In many cases, one obtains nice formulas. For instance, it is a well-known elementary exercise that if the subset of the chessboard takes the shape of a Young diagram of a partition, then the number of rook placements has a nice product formula.

In this paper, we study {\em large-scale} rook placements. We are interested in the family of rook placements when the board on which non-attacking rooks are placed grows in size in the following manner. We define a class $\mathcal{P}$ of ``nice" piecewise linear functions from $[0,1]$ to $[0,1]$ (see the exact definition in Subsection~\ref{ss:class-P}). Suppose that $f \in \mathcal{P}$ belongs to this class. We obtain a family of Young diagrams by dilating the region under the graph of $f$ from the unit square $[0,1] \times [0,1]$ to $[0,N] \times [0,N]$ for each positive integer $N$. More precisely, we let $\lambda(f, N)$ be the Young diagram with $N$ rows whose $i^{\text{th}}$ row has $(\lambda(f,N))_i := \left\lceil N \cdot f(i/N) \right\rceil$ boxes (see Subsection~\ref{ss:asymp-P} for details). Let $\RP(\lambda(f,N))$ denote the set of rook placements of the shape $\lambda(f,N)$. Our first main result of this paper, Theorem~\ref{thm:asymp-P}, says that the cardinality of $\RP(\lambda(f,N))$ behaves well asymptotically:
\[
\log \left( \#\RP(\lambda(f,N)) \right) = N \log N + B_f \cdot N + \frac{1}{2} \log N + O_f(1),
\]
for positive integers $N$ which are multiples of a certain integer depending on $f$ (see the statement of Theorem~\ref{thm:asymp-P} for details), and we establish the following integral formula for the coefficient $B_f$:
\[
B_f = \int_0^1 \log(f(x)+x-1) \, dx.
\]

Our next stop is a special subclass $\widetilde{\mathcal{P}}$ of the function class $\mathcal{P}$. We might refer to $\widetilde{\mathcal{P}}$ as the ``combinatorial" class, since it contains, rather naturally, many familiar objects from algebraic combinatorics such as Dyck paths and Motzkin paths. The class $\widetilde{\mathcal{P}}$ is a countable union of finite families of functions: $\widetilde{\mathcal{P}} = \bigcup_{k=1}^{\infty} \widetilde{\mathcal{P}}_k$ (see Section~\ref{s:tilde-P} for the precise definition). Bijective combinatorics in $\widetilde{\mathcal{P}}$ is noteworthy, and so we spend Subsection~\ref{ss:bijective-combin} discussing it. This subsection contains a bijective-combinatorics flavor which seems less analytic than its neighboring parts. For example, we provide bijective arguments resulting in Corollary \ref{cor:wt-P_k-count} which states
\[
|\widetilde{\mathcal{P}}_k| = \frac{1}{k} \binom{3k-2}{k-1}.
\]
The functions in $\widetilde{\mathcal{P}}_k$ are in one-to-one correspondence with combinatorial objects which we call ``waterfalls" (see Subsection~\ref{ss:bijective-combin} for details). Studying waterfalls yields the following curious combinatorial formula, given in Proposition~\ref{prop:curious}:
\[
\sum_{D \in \Dyck(k)} \wt(D) = \frac{1}{k} \binom{3k-2}{k-1},
\]
where $\Dyck(k)$ denotes the family of Dyck paths of length $2k$ --- lattice paths from $(0,k)$ to $(k,0)$ with $k$ unit steps to the right and $k$ unit steps down which never go below the line $X+Y = k$ --- and the {\em weight} $\wt(D)$ is defined as
\[
\wt(D) := \prod_{i=1}^{k-1} \# \left\{ j \in \mathbb{Z} \, | \, i+j > k \text{ and } (i,j) \in D \right\}.
\]

Having visited waterfalls, we proceed to Subsection~\ref{ss:precise-asymp}. Functions in the combinatorial class allow for even more precise asymptotics for the number of rook placements. Our second main result, Theorem~\ref{thm:asymp-wt-P}, provides the following asymptotic formula for $f \in \widetilde{\mathcal{P}}_k$ as follows. We have
\[
\log \left( \# \RP(\lambda(f,N)) \right) = N \log N + B_f \cdot N + \frac{1}{2} \log N + D_f + O_f(1/N),
\]
for positive integers $N \in k \mathbb{Z}$, where the coefficient $B_f$ is the same as before, and we give the following integral formula for the coefficient $D_f$:
\[
D_f := \frac{1}{2} \log(2\pi) + \frac{1}{2} \int_0^1 \frac{xf'(x) - f(x) + 1}{x(f(x)+x-1)} dx.
\]

Now that for each function $f \in \widetilde{\mathcal{P}}_k$, there are two coefficients $B_f$ and $D_f$ associated to it, one might wonder about the possible ranges of these numbers. Proposition~\ref{prop:range-B} states that
\[
-\log k - \frac{1}{k} \le B_f \le -1.
\]
Both upper bound and lower bound are tight. Each of them is attained by exactly one function in $\widetilde{\mathcal{P}}_k$.

Similarly, we have tight bounds for $D_f$. Proposition~\ref{prop:range-D} states that
\[
\frac{1}{2} \log \left( \frac{2\pi}{k} \right) \le D_f \le \frac{1}{2} \log \left( \frac{2^k \pi}{k} \right).
\]
The upper bound is attained by exactly one function in $\widetilde{\mathcal{P}}_k$. The equality cases for the lower bound is rather interesting: the lower bound is attained by a Catalan-numerous family of functions inside $\widetilde{\mathcal{P}}_k$.

Each rook placement --- or, more generally, each permutation --- comes with a certain sequence of non-negative integers called the {\em X-ray}. An object which appears in the field of discrete tomography (cf. e.g. \cite{HK99}), the X-ray of permutation has been investigated from algebraic and combinatorial points of view (cf. e.g. \cite{BF14, BMPS05}). It is related to other objects in combinatorics such as Skolem sets (cf. e.g. \cite{Nor08}) and permutohedra (cf. e.g. \cite{Pos09}). For each permutation $\pi \in S_n$, which we consider as an $n \times n$ permutation matrix, the {\em cumulative X-ray} of $\pi$ is the function $\xi_{\pi}:[0,2n] \to [0,n]$ given by
\[
\xi_{\pi}(t) := \sum_{\substack{i,j \in [n] \\ i+j \le t}} \pi_{ij},
\]
and the {\em normalized cumulative X-ray} of $\pi \in S_n$ is the function $\widetilde{\xi}_{\pi}:[0,2] \to [0,1]$ given by $\widetilde{\xi}_{\pi}(t) := \frac{1}{n} \cdot \xi_{\pi}(nt)$. Thus, the graph of the normalized cumulative X-ray is simply the graph of the cumulative X-ray rescaled from the rectangle $[0,2n] \times [0,n]$ to $[0,2] \times [0,1]$.

We have arrived at our last stop, where we consider the X-ray of a random large rook placement. For each a partition $\lambda$, we let $N \odot \lambda$ be the partition obtained from magnifying $\lambda$ by a factor of $N$ (see Section~\ref{s:rook-placements} for the precise definition). Conjecture~\ref{conj} predicts that the normalized cumulative X-ray $\widetilde{\xi}_{\pi}$ exhibits a limit shape phenomenon: for a fixed real $\varepsilon > 0$, if $\pi$ is a uniformly random rook placement of the shape $N \odot \lambda$, then
\[
\bbP\!\left( \sup_{t \in [0,2]} \left| \widetilde{\xi}_{\pi}(t) - \fm_{\lambda}(t) \right| < \varepsilon \right) \to 1,
\]
as $N \to \infty$, where $\fm_{\lambda}:[0,2] \to [0,1]$ is a certain function depending on the shape $\lambda$. Equation~(\ref{eq:limit-shape}) in Subsection~\ref{ss:cumulative-X-rays} provides a formula for this function.

Our third main result of this paper, Theorem~\ref{thm:limit-shape}, proves this conjecture in the special case when $\lambda = \square$ is a partition with one box. In other words, it says that the normalized cumulative X-ray of a uniformly random permutation, as the size of the permutation grows, exhibits a limit shape phenomenon in the above sense. We note that it is easy to compute the limit shape for the permutation case explicitly:
\[
\fm_{\square}(t) :=
\begin{cases}
\frac{t^2}{2} & \text{if } 0 < t \le 1, \text{ and} \\
-\frac{t^2}{2} + 2t - 1 & \text{if } 1 < t \le 2.
\end{cases}
\]
While Theorem~\ref{thm:limit-shape} proves Conjecture~\ref{conj} for only one very special case, we hope that one proof technique is applicable, perhaps with some more work, for other shapes $\lambda$ as well.

We remark that since rook placements can be considered as permutations, our work in this paper is closely related to an active and exciting field of research on large permutations and ``permutons" (cf. e.g. \cite{HKMRS13, AM14, GGKK15, GHK3L17, KKRW20}). For example, our construction of the normalized cumulative X-ray is reminiscent of that of permutons. It would be interesting, in the author's opinion, to see how tools from the permuton literature can be applied to better understand large rook placements.

\subsection*{Outline} In Section~\ref{s:rook-placements}, we give some definitions and present some elementary facts about rook placements. Section~\ref{s:P} focuses on the class $\mathcal{P}$ of ``nice" piecewise linear functions. It contains Theorem~\ref{thm:asymp-P}, our first main result. Section~\ref{s:tilde-P} focuses on the ``combinatorial" class $\widetilde{\mathcal{P}}$. We discuss some bijective combinatorics in Subsection~\ref{ss:bijective-combin}. We establish Theorem~\ref{thm:asymp-wt-P}, our second main result, in Subsection~\ref{ss:precise-asymp}. We give some properties of the coefficient $D_f$ in Subsection~\ref{ss:prop-D}. We prove inequalities on the coefficients $B_f$ and $D_f$ in Subsection~\ref{ss:bounds-B-D}. In Section~\ref{s:X-rays}, we discuss probabilities and X-rays. It contains Conjecture~\ref{conj}. We deduce this conjecture in the special case of random permutations from Theorem~\ref{thm:limit-shape}, our third main result, in Subsection~\ref{ss:limit-random-permutation}.

\bigskip

\section{Rook Placements}\label{s:rook-placements}
For each positive integer $n$, let $S_n$ denote the set of permutations of $[n] := \{1, 2, \ldots, n\}$. We think of a permutation $\pi \in S_n$ as an $n \times n$ matrix (``permutation matrix")
\[
\pi = \begin{bmatrix}
\pi_{11} & \pi_{12} & \cdots & \pi_{1n} \\
\pi_{21} & \pi_{22} & \cdots & \pi_{2n} \\
\vdots & \vdots & \ddots & \vdots \\
\pi_{n1} & \pi_{n2} & \cdots & \pi_{nn}
\end{bmatrix},
\]
where each entry $\pi_{ij}$ is either $0$ or $1$, each row has exactly one $1$, and each column has exactly one $1$. A {\em partition} is a finite sequence of weakly decreasing positive integers. A partition $\lambda$ is said to have {\em (exactly) $n$ parts} if the length of $\lambda$, as a finite sequence, is $n$. Let us denote by $\Par$ the set of all partitions. By convention, we also include the empty partition $[\hphantom{0}]$ in $\Par$. Consider the set
\[
\cB_n := \left\{ [\lambda_1, \lambda_2, \ldots, \lambda_n] \in \Par | \, n = \lambda_1 \ge \lambda_2 \ge \cdots \ge \lambda_n > 0 \right\}.
\]
In other words, $\cB_n$ is the set of partitions $\lambda$ with exactly $n$ parts such that $\lambda_1 = n$. Given a partition $\lambda \in \cB_n$, a {\em rook placement} of the shape $\lambda$ is a permutation $\pi \in S_n$ such that
\begin{center}
    for any $i, j \in [n]$, if $j > \lambda_{n+1-i}$, then $\pi_{ij} = 0$.
\end{center}
We use the notation $\RP(\lambda)$ to denote the set of rook placements of the shape $\lambda$. The cardinality of $\RP(\lambda)$ has a well-known and easy-to-prove formula: for any $\lambda \in \cB_n$,
\begin{equation}\label{eq:RP-prod}
    \#\RP(\lambda) = \prod_{i=1}^n (\lambda_i - (n-i)).
\end{equation}
One particular point to notice about the product formula above that is particularly beautiful, in the author's opinion, is that the formula holds {\em even when} there are no rook placements of the shape $\lambda$. In other words, for partitions $\lambda \in \cB_n$ such that $\RP(\lambda) = \varnothing$, the right-hand side of the formula becomes zero (not some negative integer).

We define $\cD_n := \{\lambda \in \cB_n \, | \, \RP(\lambda) \text{ is not empty.}\}$. It is well-known that for $\lambda \in \cB_n$, the partition $\lambda$ belongs to $\cD_n$ if and only if for all $i \in [n]$, we have $\lambda_i \ge n+1-i$. It is also well-known that $|\cB_n|$ is the central binomial coefficient $\binom{2n-2}{n-1}$ and that $|\cD_n|$ is the $n^{\text{th}}$-Catalan number $C_n := \frac{1}{n+1}\binom{2n}{n}$.

For each non-negative integer $n$, let $\Par(n)$ denote the set of all partitions $\lambda$ such that the sum of all parts of $\lambda$ is $n$. Now we describe how we {\em dilate} partitions. Suppose that $\lambda \in \Par(n)$ and let $m$ be a positive integer. We define the partition $m \odot \lambda \in \Par(m^2n)$ as follows. Imagine starting with the Young diagram of $\lambda$, and then replacing each of the $n$ boxes of $\lambda$ with an $m \times m$ array of boxes. The resulting diagram is the Young diagram of $m \odot \lambda$.

The following formula for the size of $\RP(m \odot \lambda)$ is an immediate consequence (and also a mild generalization) of Equation~(\ref{eq:RP-prod}).

\begin{proposition}
Let $m$ and $n$ be positive integers. For any partition $\lambda = [\lambda_1, \lambda_2, \ldots, \lambda_n] \in \cB_n$, we have $m \odot \lambda \in \cB_{mn}$ and
\[
\#\RP(m \odot \lambda) = m!^n \cdot \prod_{i=1}^n \binom{m(\lambda_i - (n-i))}{m}.
\]
Here, the {\em binomial coefficient} is defined for $a \in \mathbb{Z}$ and $b \in \mathbb{Z}_{\ge 1}$ as $\binom{a}{b} := \frac{a(a-1)\cdots (a-b+1)}{b!}$.
\end{proposition}

It is easy to see that for $\lambda \in \cB_n$, the partition $\lambda$ belongs to $\cD_n$ if and only if for any positive integer $m$, the partition $m \odot \lambda$ belongs to $\cD_{mn}$.

\bigskip

\section{The class $\mathcal{P}$ of piecewise linear functions}\label{s:P}
\subsection{The functions and their lofts}\label{ss:class-P}
Consider the class $\mathcal{P}$ of functions $f:[0,1] \to [0,1]$ with the following properties:
\begin{itemize}
\item $f$ is weakly decreasing,
\item $f$ is piecewise linear with a finite number of non-differentiable points,
\item all the non-differentiable points of $f$ are rational numbers in $[0,1]$,
\item there exists $\varepsilon > 0$ such that for any $0 \le x < \varepsilon$, we have $f(x) = 1$,
\item $f(1) > 0$, and
\item for any $a \in (0,1)$, we have $\lim_{x \searrow a} f(x) > 1 - a$.
\end{itemize}

\begin{example}\label{ex:irr-slope}
An example of a function in $\mathcal{P}$ is the following function $f:[0,1] \to [0,1]$ given by
\[
f(x) := \begin{cases}
1 & \text{if } x < \frac{1}{2}, \\
\frac{1}{\sqrt{2}} & \text{if } x = \frac{1}{2}, \\
\frac{1}{\sqrt{2}} - \frac{x}{\sqrt{7}} & \text{if } x > \frac{1}{2}.
\end{cases}
\]
It is straightforward to check that all the conditions for functions to be in $\mathcal{P}$ are satisfied. Note that while we require the non-differentiable points to be rational numbers in $[0,1]$, it is fine for the {\em values} of the function at the non-differentiable points to be irrational. In our example here, the value of the function at the non-differentiable point $1/2$ is $1/\sqrt{2}$, which is irrational. Moreover, it is also fine for the slope of some piece of the function to be irrational. In our example here, the slope of the function when $x \in (1/2,1]$ is $-1/\sqrt{7}$, which is irrational.
\end{example}

\begin{example}
Here we present a non-example. A function that does {\em not} belong to $\mathcal{P}$ is the function $g:[0,1] \to [0,1]$ given by
\[
g(x) := \begin{cases}
1 & \text{if } x \le \frac{1}{2}, \\
\frac{1}{2} & \text{if } x > \frac{1}{2}.
\end{cases}
\]
Note that even though $g(x) > 1 - x$ for all $x \in (0,1]$, the limit $\lim_{x \searrow (1/2)} g(x) = 1/2$. This violates the last condition for functions to belong to $\mathcal{P}$.
\end{example}

The following proposition gives some basic properties of functions in $\mathcal{P}$. These properties can be proved immediately from the definition of $\mathcal{P}$, so we omit the proof.

\begin{proposition}\label{prop:basic-P}
Let $f \in \mathcal{P}$. Then,
\begin{itemize}
\item[(a)] for every $x \in [0,1]$, we have $0 < f(x) \le 1$.
\item[(b)] for every $a \in (0,1]$, we have $f(a) > 1-a$ and $\lim_{x \nearrow a} f(x) > 1-a$.
\item[(c)] for every $\varepsilon \in (0,1]$, there exists $\delta > 0$ such that for every $x \in [\varepsilon,1]$, we have the inequality $f(x) + x - 1 > \delta$.
\end{itemize}
\end{proposition}

Each function $f \in \mathcal{P}$ comes with a useful quantity we call the {\em loft} of $f$ defined as follows.

\begin{definition}\label{defn:loft}
For each function $f \in \mathcal{P}$, define the {\em loft} of $f$ as
\[
\loft(f) := \sup \left\{ \varepsilon \in [0,1] : \forall x \in [0,\varepsilon], f(x) = 1, \text{ and } \, \forall x \in (\varepsilon,1], f(x) > 1 - x + \varepsilon \right\}.
\]
\end{definition}

The following proposition gives some basic properties of the loft of a function in $\mathcal{P}$.

\begin{proposition}\label{prop:basic-loft}
Let $f \in \mathcal{P}$. Then,
\begin{itemize}
\item[(a)] its loft is strictly positive: $0 < \loft(f) \le 1$.
\item[(b)] for every real number $x$ such that $0 \le x \le \loft(f)$, we have $f(x) = 1$.
\item[(c)] for every real number $x$ such that $\loft(f) \le x \le 1$, we have $f(x) + x - 1 \ge \loft(f)$.
\item[(d)] for every $x \in [0,1]$, we have
\[
x \ge f(x) + x - 1 \ge \min \left\{ x, \loft(f) \right\}.
\]
\end{itemize}
\end{proposition}
\begin{proof}
Consider any function $f \in \mathcal{P}$. Let $\mathcal{X}$ denote the set from Definition~\ref{defn:loft}:
\[
\mathcal{X} := \left\{ \varepsilon \in [0,1] : \forall x \in [0,\varepsilon], f(x) = 1, \text{ and } \, \forall x \in (\varepsilon,1], f(x) > 1 - x + \varepsilon \right\}.
\]

\smallskip

\textbf{(a)} It suffices to show that $\mathcal{X} \cap (0,1] \neq \varnothing$. Since $f \in \mathcal{P}$, there exists some $a > 0$ such that $f(a) = 1$. By Proposition~\ref{prop:basic-P}(c), there exists $b > 0$ such that for every $x \in [a,1]$, we have $f(x) + x - 1 > b$. Take $c := \min \{a, b\} \in (0,1]$. We claim that $c \in \mathcal{X}$.

First, for any $x \in [0,c]$, we have $1 \ge f(x) \ge f(c) \ge f(a) = 1$ and so $f(x) = 1$. Second, suppose $x \in (c,1]$. If $c < x \le a$, then $f(x) = 1 > 1 - x + c$. If $x > a$, then $f(x) + x - 1 > b \ge c$ and thus $f(x) > 1 - x + c$. This shows that $c \in \mathcal{X}$.

\smallskip

\textbf{(b)} It suffices to show that $f(\loft(f)) = 1$. If $\loft(f) \in \mathcal{X}$, we are done. If $\loft(f) \notin \mathcal{X}$, then for any positive integer $n$, there exists $a_n \in \mathcal{X}$ with $\loft(f) - \frac{1}{n} < a_n < \loft(f)$. Since $\loft(f) > a_n$, we have that
\[
f(\loft(f)) > 1 - \loft(f) + a_n > 1 - \frac{1}{n}.
\]
Since $n$ is arbitrary, we have $f(\loft(f)) = 1$.

\smallskip

\textbf{(c)} This is similar to part (b). If $\loft(f) \in \mathcal{X}$ and $x > \loft(f)$, then by the definition of $\mathcal{X}$, we have $f(x) > 1 - x + \loft(f)$. If $\loft(f) \in \mathcal{X}$ and $x = \loft(f)$, then $f(x) + x - 1 = f(\loft(f)) + \loft(f) - 1 = \loft(f)$.

On the other hand, if $\loft(f) \notin \mathcal{X}$, then for any positive integer $n$, there exists $a_n \in \mathcal{X}$ with $\loft(f) - \frac{1}{n} < a_n < \loft(f)$. For every real number $x \ge \loft(f)$, we then have $x > a_n$, and thus
\[
f(x) > 1 - x + a_n > 1 - x + \loft(f) - \frac{1}{n}.
\]
Since $n$ is arbitrary, we have $f(x) \ge 1 - x + \loft(f)$.

\smallskip

\textbf{(d)} This part follows from parts (b) and (c).
\end{proof}

Proposition~\ref{prop:basic-loft}(d) is an analytically useful property of the loft of a function in $\mathcal{P}$. It says roughly that once $x \in [0,1]$ is far enough from $0$, the point $(x,f(x))$ on the graph of the function is far enough from the line $X+Y = 1$.

\subsection{An asymptotic formula for the number of rook placements for functions in $\mathcal{P}$}\label{ss:asymp-P} Suppose that a function $f \in \mathcal{P}$ is given. Let $\rho_1, \rho_2, \ldots, \rho_m$ be the non-differentiable points of $f$ inside the open interval $(0,1)$, listed in increasing order. (Here $m$ is a non-negative integer. We use the convention that $m = 0$ if and only if $f$ is differentiable on $(0,1)$, which is when $f(x) = 1$ for all $0 \le x < 1$.) For convenience, we define $\rho_0 := 0$ and $\rho_{m+1} := 1$. Note that for each $i \in [m+1]$, the function $f$ is linear on the open interval $(\rho_{i-1}, \rho_i)$. Define the function $f_i:[\rho_{i-1},\rho_i] \to \mathbb{R}$ to be the unique linear extension of $f|_{(\rho_{i-1}, \rho_i)}$ from $(\rho_{i-1}, \rho_i)$ to $[\rho_{i-1}, \rho_i]$. There exist a non-positive real number $\mu_i$ and a real number $\beta_i$ such that $f_i(x) = \mu_i x + \beta_i$ for $x \in [\rho_{i-1}, \rho_i]$.

Note that since we define $f_i$ on the {\em closed} interval $[\rho_{i-1}, \rho_i]$, the functions $f_i$ and $f$ might have different values at $\rho_{i-1}$ and $\rho_i$. On the other hand, the two functions agree in the interior of the interval. Note also that $f_1(x) \equiv 1$ (i.e., $\mu_1 = 0$ and $\beta_1 = 1$).

Take any positive integer $N$ such that $N\rho_i \in \mathbb{Z}$ for every $i$. We define the partition $\lambda(f,N) \in \Par$ to be the partition with exactly $N$ parts whose $i^{\text{th}}$ part is given by
\[
\left(\lambda(f,N)\right)_i := \left\lceil N \cdot f(i/N) \right\rceil.
\]
Our definition of $\mathcal{P}$ guarantees that, as one may readily verify, $\lambda(f,N) \in \cD_N$; in other words, $\RP(\lambda(f, N))$ is always non-empty.

Our goal of this subsection is to compute an asymptotic formula for $\#\RP(\lambda(f,N))$ of the form
\[
\log \left( \#\RP(\lambda(f,N)) \right) = A_f \cdot N \log N + B_f \cdot N + C_f \cdot \log N + O_f(1),
\]
for positive integers $N$ such that $N\rho_i \in \mathbb{Z}$ for every $i$. Here, the notation $O_f$ means that the implicit constant depends only on the function $f$.

By Equation~(\ref{eq:RP-prod}), we can write
\begin{equation}\label{eq:RfN-appears}
    \log\left( \#\RP(\lambda(f,N)) \right) = \left( \sum_{0 < n \le N} \log \left( N f(n/N) + n - N \right) \right) + R(f,N),
\end{equation}
where $R(f,N)$ is the discrepancy from rounding:
\begin{equation}\label{eq:RfN-defn}
    R(f,N) := \sum_{0 < n \le N} \log \left( \frac{\left\lceil N f(n/N) \right\rceil + n - N}{N f(n/N) + n - N} \right).
\end{equation}

\begin{proposition}\label{prop:RfN-is-Of1}
We have $R(f,N) \ge 0$ and $R(f,N) = O_f(1)$.
\end{proposition}
\begin{proof}
The first item $R(f,N) \ge 0$ is clear from Equation~(\ref{eq:RfN-defn}). We proceed to show that $R(f,N) = O_f(1)$. Notice that for $0 < n < \rho_1 N$, we have $f(n/N) = 1$. Therefore, we can write
\begin{equation}
    R(f,N) = \log \left( \frac{\left\lceil Nf(\rho_1) \right\rceil + \rho_1 N - N}{Nf(\rho_1) + \rho_1 N - N} \right) + \sum_{\rho_1 N < n \le N} \log \left( \frac{\left\lceil Nf(n/N) \right\rceil + n - N}{Nf(n/N) + n - N} \right).
\end{equation}
Since $\rho_1 \in (0,1]$, by Proposition~\ref{prop:basic-P}(c), there exists $a > 0$ such that for every $x \in [\rho_1, 1]$ we have $f(x) + x - 1 > a$. Using this with the inequality $\log(\left\lceil x \right\rceil) - \log(x) < 1/x, \forall x > 0$, we obtain
\begin{align*}
    R(f,N) &< \frac{1}{N(f(\rho_1)+\rho_1-1)} + \sum_{\rho_1 N < n \le N} \frac{1}{N \cdot \left(f(n/N) + (n/N) - 1\right)} \\
    &\le \frac{1}{N(f(\rho_1) + \rho_1 - 1)} + \sum_{\rho_1 N < n \le N} \frac{1}{N \cdot a} \\
    &= \frac{1}{f(\rho_1) + \rho_1 - 1} \cdot \frac{1}{N} + \frac{1-\rho_1}{a}.
\end{align*}
Since $\rho_1$ and $a$ depend only on $f$ (and not $N$), the quantity above is $O_f(1)$.
\end{proof}

Recall that the function $f$ is linear on each {\em open} interval $(\rho_{i-1}, \rho_i)$, while, at each $\rho_i$, there might be a ``jump." For instance, in Example~\ref{ex:irr-slope} above, the three values $\lim_{x \nearrow \rho_1} f(x)$, $f(\rho_1)$, and $\lim_{x \searrow \rho_1} f(x)$ are all different. It is not hard to see, however, that these possible jumps do not have a huge effect on the summation in Equation~(\ref{eq:RfN-appears}):
\begin{align}
    &\sum_{0 < n \le N} \log \left( N f(n/N) + n - N \right) \\
    &= \sum_{i=1}^{m+1} \left[ \sum_{\rho_{i-1} N < n \le \rho_i N} \log((\mu_i+1) n + (\beta_i-1) N) \right] + O_f(1).\label{eq:effect-from-jumps}
\end{align}
Combining this with Equation~(\ref{eq:RfN-appears}) and Proposition~\ref{prop:RfN-is-Of1}, we obtain
\begin{equation}\label{eq:RP-sum-log}
    \log\left( \#\RP\left( \lambda(f, N) \right) \right) = \sum_{i=1}^{m+1} \left[ \sum_{\rho_{i-1} N < n \le \rho_i N} \log((\mu_i+1) n + (\beta_i-1) N) \right] + O_f(1).
\end{equation}
We break the outer summation on the right-hand side above into when $i = 1$ and when $i \ge 2$. When $i = 1$, we have, by Stirling's formula,
\begin{align}
&\sum_{\rho_{i-1} N < n \le \rho_i N} \log((\mu_i+1)n+(\beta_i-1)N) = \log((\rho_1 N)!) \label{eq:rho_1N} \\
&= \rho_1 \cdot N\log N + (\rho_1 \log \rho_1 - \rho_1) \cdot N + \frac{1}{2} \log N + O_f(1).
\end{align}

When $2 \le i \le m+1$, observe that the function
\[
x \mapsto \log((\mu_i + 1) x + (\beta_i - 1)N)
\]
is well-defined on the whole closed interval $[\rho_{i-1} N, \rho_i N]$. Using the Euler-Maclaurin summation formula (cf. \cite[Appendix~B]{MV07}) with this function, we write
\begin{equation}
    \sum_{\rho_{i-1} N < n \le \rho_i N} \log((\mu_i+1)n+(\beta_i-1)N) = \int_{\rho_{i-1} N}^{\rho_i N} \log((\mu_i+1) x + (\beta_i-1) N) dx + O_f(1).
\end{equation}

By the change of variables $x \mapsto N \cdot x'$, we have
\begin{align}
    &\int_{\rho_{i-1} N}^{\rho_i N} \log((\mu_i+1) x + (\beta_i-1) N) dx \\
    &= (\rho_i - \rho_{i-1}) \cdot N \log N + \left\{ \int_{\rho_{i-1}}^{\rho_i} \log(f(x)+x-1) \, dx \right\} \cdot N. \label{eq:change-of-var}
\end{align}

The following is the main theorem of this section.
\begin{theorem}\label{thm:asymp-P}
Let $f \in \mathcal{P}$. Let $\rho_0, \rho_1, \ldots, \rho_{m+1}$ be as defined above. We have
\[
\log \left( \#\RP(\lambda(f,N)) \right) = N \log N + B_f \cdot N + \frac{1}{2} \log N + O_f(1),
\]
for positive integers $N$ such that $N\rho_i \in \mathbb{Z}$ for every $i$, where
\[
B_f = \int_0^1 \log(f(x)+x-1) \, dx.
\]
\end{theorem}

Note that the integral is improper at $x = 0$. We interpret it as
\[
\lim_{\varepsilon \searrow 0} \int_{\varepsilon}^1 \log(f(x)+x-1) \, dx.
\]

\begin{proof}[Proof of Theorem~\ref{thm:asymp-P}]
This result is immediate from combining Equations~(\ref{eq:rho_1N})-(\ref{eq:change-of-var}) and observing that $\lim_{\varepsilon \searrow 0} \int_{\varepsilon}^{\rho_1} \log(f(x)+x-1) \, dx = \lim_{\varepsilon \searrow 0} \int_{\varepsilon}^{\rho_1} \log(x) \, dx = \rho_1 \log \rho_1 - \rho_1$.
\end{proof}

\subsection{Properties of $B_f$} Let $p \ge 1$ be any positive real number. By considering the $L^p$-norm of functions in $\mathcal{P}$, we make the class $\mathcal{P}$ a metric space. Let us denote this metric space by $(\mathcal{P}, L^p)$. A technical remark is that in the construction of $(\mathcal{P}, L^p)$, we identify any two functions $f, g \in \mathcal{P}$ for which $\|f-g\|_p = 0$. An element in $(\mathcal{P}, L^p)$ is an equivalence class of functions.

Nevertheless, it is easy to see that the map $f \mapsto B_f$ induces a well-defined map
\[
\boldsymbol{B}: (\mathcal{P}, L^p) \to \mathbb{R}.
\]
Let $(\mathbb{R}, \operatorname{Euclid})$ denote the set of real numbers equipped with the usual Euclidean metric. We have the following topological property of $\boldsymbol{B}$, when considered as a function from $(\mathcal{P}, L^p)$ to $(\mathbb{R}, \operatorname{Euclid})$.

\begin{proposition}\label{prop:B-discontinuous}
The map $\boldsymbol{B}:(\mathcal{P}, L^p) \to (\mathbb{R}, \operatorname{Euclid})$ is discontinuous everywhere on $\mathcal{P}$.
\end{proposition}
\begin{proof}
Let $f \in \mathcal{P}$ be an arbitrary function. For each positive integer $n$ such that $n^{-1} + 2^{-n^2} < \loft(f)$, define
\[
g_n(x) := \begin{cases}
f(x) & \text{ if } x \le 1 - \frac{1}{n}, \\
1 - x + 2^{-n^2} & \text{ if } x > 1 - \frac{1}{n}.
\end{cases}
\]
Observe that $g_n$ converges to $f$ in $L^p$, as $n \to \infty$. However, by the triangle inequality, we have
\[
|B_{g_n}| = \left| \int_0^1 \log(g_n(x) + x - 1) \, dx \right| \ge |I_1| - |I_2| - |I_3|,
\]
where
\[
I_1 := \int_{1-\frac{1}{n}}^1 \log(g_n(x) + x - 1) \, dx,
\]
\[
I_2 := \int_0^1 \log(f(x) + x - 1) \, dx,
\]
and
\[
I_3 := \int_{1-\frac{1}{n}}^1 \log(f(x) + x - 1) \, dx.
\]
Note that $|I_1| = (\log 2) \cdot n$, $|I_2| = |B_f|$, and $|I_3| \to 0$, as $n \to \infty$. Hence, $\{g_n\}$ is a sequence of functions converging to $f$ in $L^p$, but $|B_{g_n}| \to \infty$, as $n \to \infty$.
\end{proof}

The following proposition is clear from the integral formula of $B_f$.

\begin{proposition}\label{prop:from-int-form-B}
\begin{itemize}
\item[(a)] For every function $f \in \mathcal{P}$, we have $B_f \le -1$. The upper bound is tight. The equality is attained if and only if $f(x) = 1$ for every $x \in [0,1)$.
\item[(b)] For any functions $f, g \in \mathcal{P}$ such that $\forall x \in [0,1]$, $f(x) \ge g(x)$, we have $B_f \ge B_g$.
\end{itemize} 
\end{proposition}

\bigskip

\section{The class $\widetilde{\mathcal{P}}$ of piecewise linear functions}\label{s:tilde-P}
For each positive integer $k$, define $\widetilde{\mathcal{P}}_k$ to be the class of functions $f \in \mathcal{P}$ which satisfy the following additional properties:
\begin{itemize}
\item for each $i \in [k]$, we have $k \cdot f(i/k) \in \mathbb{Z}$,
\item $f$ is upper-semicontinuous, and
\item for each $i \in [k]$, the restriction $f|_{((i-1)/k, i/k)}$ is linear with a non-positive integer slope.
\end{itemize}

We let $\widetilde{\mathcal{P}} := \bigcup_{k=1}^{\infty} \widetilde{\mathcal{P}}_k$. One important property about function $f \in \widetilde{\mathcal{P}}$ is that for every $a \in (0,1]$, we have $\lim_{x \nearrow a} f(x) = f(a)$.

\subsection{Bijective Combinatorics in $\widetilde{\mathcal{P}}_k$}\label{ss:bijective-combin} It is easy to see that $\widetilde{\mathcal{P}}_k$ is a finite set. In fact, its size has a nice product formula.

\begin{proposition}\label{prop:bij-P-T-G}
For each positive integer $k \ge 2$, the sizes of the following sets are equal:
\begin{itemize}
\item the class $\widetilde{\mathcal{P}}_k$,
\item the set $\mathcal{T}_k$ of $(2k-3)$-tuples $\left(y_1, y_2, \ldots, y_{2k-3}\right)$ of non-negative integers such that for each $m \in [2k-3]$, we have
\[
y_1 + y_2 + \cdots + y_m \le \left\lceil \frac{m}{2} \right\rceil,
\]

\item the set $\mathcal{G}_k$ of lattice paths from $(0,0)$ to $(2k-1,k-1)$ which are contained in the half-plane $\{(x,y) : x \ge 2y\}$.
\end{itemize}
\end{proposition}
\begin{proof} We construct the following bijections.

First, $\widetilde{\mathcal{P}}_k \to \mathcal{T}_k$. Given $f \in \widetilde{\mathcal{P}}_k$, we define $(y_1, y_2, \ldots, y_{2k-3})$ as follows. For each $h$ such that $1 \le h \le k-1$, let
\[
y_{2h-1} := k \cdot \left( \lim_{x \searrow h/k} f(x) - f\!\left( \frac{h+1}{k} \right)\right),
\]
and for each $\ell$ such that $1 \le \ell \le k-2$, let
\[
y_{2\ell} := k \cdot \left( f\!\left( \frac{\ell+1}{k} \right) - \lim_{x \searrow (\ell+1)/k} f(x) \right).
\]

Second, $\mathcal{T}_k \to \mathcal{G}_k$. Send the tuple $(y_1, \ldots, y_{2k-3}) \in \mathcal{T}_k$ to the path
\[
E^2 N^{y_1} E N^{y_2} E N^{y_3} \cdots E N^{y_{2k-3}} E N^{k-1-y_1-y_2-\cdots-y_{2k-3}},
\]
where $E$ denotes the step $(1,0)$ and $N$ denotes the step $(0,1)$.
\end{proof}

\begin{corollary}\label{cor:wt-P_k-count}
For every positive integer $k$, we have
\[
|\widetilde{\mathcal{P}}_k| = \frac{1}{k} \binom{3k-2}{k-1}.
\]
\end{corollary}
Note that this sequence appears as \verb!A006013! on the OEIS \cite{OEIS}.
\begin{proof}[Proof of Corollary~\ref{cor:wt-P_k-count}]
With Proposition~\ref{prop:bij-P-T-G}, it suffices to show that $|\mathcal{G}_k| = \frac{1}{k} \binom{3k-2}{k-1}$. This is well-known: see, for example, the sequence $\{b_n\}$ in the work of Gessel and Xin \cite[Section~3]{GX06}, or references in \verb!A006013! on the OEIS \cite{OEIS}.
\end{proof}

In the following discussion, by a {\em lattice path}, we mean the image of an injective continuous function $[0,1] \to \mathbb{R}^2$ (under the usual Euclidean topology for both spaces) that is also a finite union of segments $s_1, \ldots, s_n$ such that both end points of each $s_i$ are lattice points.

Functions in $\widetilde{\mathcal{P}}_k$ can be seen as lattice paths, by dilating their graphs and then adding vertical segments. Formally, suppose a function $f \in \widetilde{\mathcal{P}}_k$ is given. We first dilate the graph of $f$ into the square $[0,k] \times [0,k]$ by
\[
\Gamma_f := \left\{ (x,y) \in \mathbb{R} \times \mathbb{R} : \frac{y}{k} = f \! \left( \frac{x}{k} \right) \right\} \subseteq [0,k] \times [0,k].
\]
Then, we take the closure $\overline{\Gamma}_f$ of $\Gamma_f$ with respect to the usual Euclidean topology on $\mathbb{R}^2$. Note that the closure simply adds a finite number of points into the set $\Gamma_f$. Then, our path $\Path(f)$ is given by
\[
\Path(f) := \left\{ (x,y) \in \mathbb{R}^2 \, | \, \text{there exist } y_1, y_2 \text{ such that } y_1 \le y \le y_2 \text{ and } (x,y_1), (x,y_2) \in \overline{\Gamma}_f \right\}.
\]
For any $f \in \widetilde{\mathcal{P}}_k$, the path $\Path(f)$ is a lattice path with endpoints $(0,k)$ and $(k,0)$. Let us define $\mathcal{L}_k := \Path(\widetilde{\mathcal{P}}_k)$.

The map $\Path: \widetilde{\mathcal{P}}_k \to \mathcal{L}_k$ is a bijection. To go back from lattice paths to functions, consider the map $\Func: \mathcal{L}_k \to \widetilde{\mathcal{P}}_k$ given by
\[
\left(\Func(\gamma)\right)(x) := \sup \left\{ y : (kx, ky) \in \gamma \right\},
\]
for any $x \in [0,1]$, for any $\gamma \in \mathcal{L}_k$. This map simply shrinks the path back and then removes vertical segments. It is straightforward to see that $\Path$ and $\Func$ are inverses.

If we think of paths in $\mathcal{L}_k$ as going from $(0,k)$ to $(k,0)$, then they are exactly the lattice paths $\gamma$ with the following properties:
\begin{itemize}
\item the path $\gamma$ starts at $(0,k)$ and ends at $(k,0)$,
\item each step is either $(1,-\ell)$ for $\ell \in \mathbb{Z}_{\ge 0}$ or $(0,-1)$,
\item the path $\gamma$ intersects with the diagonal $X+Y = k$ exactly at its two endpoints.
\end{itemize}

This class $\mathcal{L}_k$ of lattice paths contains many familiar paths in algebraic combinatorics such as Dyck paths and Motzkin paths. It is also closely related to plane $S$-trees, parenthesizations, and dissections of a convex polygon. See Stanley's text \cite[Chapter~6]{EC2} for details.

For the following discussion, a {\em Dyck path} from $(0,k)$ to $(k,0)$ is a lattice path starting from $(0,k)$, using steps $(1,0)$ and $(0,-1)$, and ending at $(k,0)$ that never crosses (but might touch) the line $X+Y = k$. We let $\Dyck(k)$ denote the set of Dyck paths from $(0,k)$ to $(k,0)$. 

Since paths in $\mathcal{L}_k$ can intersect with the line $X+Y = k$ only at the two endpoints $(0,k)$ and $(k,0)$, the Dyck paths in $\mathcal{L}_k$ are in bijection with the Dyck paths from $(1,k)$ to $(k,1)$. Note that Dyck paths in $\mathcal{L}_k$ correspond (under $\Func: \mathcal{L}_k \to \widetilde{\mathcal{P}}_k$) to piecewise constant functions in $\widetilde{\mathcal{P}}_k$. We have thus obtained one trivial embedding of a Catalan-numerous family into $\widetilde{\mathcal{P}}$.

\begin{proposition}
For each positive integer $k$, the number of piecewise constant functions in $\widetilde{\mathcal{P}}_k$ is exactly the $(k-1)^{\text{st}}$ Catalan number
\[
C_{k-1} = \frac{1}{k} \binom{2k-2}{k-1}.
\]
\end{proposition}

Similarly, we have a Motzkin-numerous class of functions in $\widetilde{\mathcal{P}}$ as follows.

\begin{proposition}
For each positive integer $k \ge 2$, define the subset $\mathcal{MO}_k \subseteq \widetilde{\mathcal{P}}_k$ to be the class of all functions $f \in \widetilde{\mathcal{P}}_k$ which satisfy the following conditions:
\begin{itemize}
\item each linear piece of $f$ either is constant or has slope $-1$,
\item for each non-differentiable point $a \in (0,1)$ of $f$, we have
\[
k \cdot f(a) \equiv k \cdot \lim_{x \searrow a} f(x) \equiv k \cdot (1 - a) \pmod{2},
\]
\item the number $k \cdot f(1)$ is an even integer.
\end{itemize}
Then, the size of $\mathcal{MO}_k$ is the $(k-2)^{\text{nd}}$ Motzkin number $M_{k-2}$. (For more details about the Motzkin numbers, we recommend Stanley's text \cite[Exercises~6.37~and~6.38]{EC2}.)
\end{proposition}

If we drop the last two conditions about parity, we obtain Schr\"{o}der numbers.

\begin{proposition}
For each positive integer $k$, define the subset $\mathcal{SC}_k \subseteq \widetilde{\mathcal{P}}_k$ to be the class of all functions $f \in \widetilde{\mathcal{P}}_k$ such that each linear piece of $f$ either is constant or has slope $-1$. Then, the size of $\mathcal{SC}_k$ is the $(k-1)^{\text{st}}$ Schr\"{o}der number $r_{k-1}$. (For more details about the Schr\"{o}der numbers, we recommend Stanley's text \cite[Section~6.2~and~Exercises~6.39]{EC2}.)
\end{proposition}

There is another embedding of a Catalan-numerous family in $\widetilde{\mathcal{P}}$. The following proposition is observed and proved by Alex Postnikov.

\begin{proposition}\label{prop:continuous-family}
Let $k$ be a positive integer. The number of continuous functions in $\widetilde{\mathcal{P}}_k$ is exactly the $k^{\text{th}}$ Catalan number
\[
C_k = \frac{1}{k+1} \binom{2k}{k}.
\]
\end{proposition}
\begin{proof}
We construct an explicit bijection from the set of continuous functions in $\widetilde{\mathcal{P}}_k$ to $\Dyck(k)$, the set of Dyck paths from $(0,k)$ to $(k,0)$. Note that for each continuous function $f \in \widetilde{\mathcal{P}}_k$, the graph of $f$, after dilating to $[0,k] \times [0,k]$, is a continuous lattice path from $(0,k)$ to $(k,0)$ without a vertical step. We can write this path as the union of $k$ segments $\bigcup_{i=1}^k \gamma_i$, where
\[
 \gamma_i := \left[ \left(i-1, k \cdot f \left( \frac{i-1}{k} \right)\right), \left(i, k \cdot f \left( \frac{i}{k} \right)\right) \right]
\]
Replace each segment $\gamma_i$ with an $L$-shaped broken segment with the same endpoints:
\[
\gamma'_i := \left[ \left(i-1, k \cdot f \left( \frac{i-1}{k} \right)\right), \left(i-1, k \cdot f \left( \frac{i}{k} \right)\right) \right] \cup \left[ \left(i-1, k \cdot f \left( \frac{i}{k} \right)\right), \left(i, k \cdot f \left( \frac{i}{k} \right)\right) \right].
\]
Then, the union $\bigcup_{i=1}^k \gamma'_i$ is the desired Dyck path in $\Dyck(k)$.
\end{proof}

\begin{remark}\label{rmk:f(1)-is-1/k-family}
From Proposition~\ref{prop:continuous-family}, we quickly obtain yet another Catalan-numerous family. The number of continuous functions $f \in \widetilde{\mathcal{P}}_k$ for which $f(1) = 1/k$ is the $(k-1)^{\text{st}}$ Catalan number $C_{k-1}$. This is because under the bijection in the proof of Proposition~\ref{prop:continuous-family}, these functions become Dyck paths in $\Dyck(k)$ that visit $(k-1,1)$.
\end{remark}

The bijection in the proof of Proposition~\ref{prop:continuous-family} that sends continuous functions to Dyck paths might be extended to the whole $\widetilde{\mathcal{P}}_k$. The image of the extended map can be understood as Dyck paths with certain marks on vertical segments. These are combinatorial objects which we call {\em waterfalls}.

\begin{definition}
A {\em waterfall} of size $k$ is a Dyck path $\gamma \in \Dyck(k)$ together with a choice of coloring of every unit segment in $\gamma$ so that each segment is colored one of either green or blue with the following rules:
\begin{itemize}
    \item every horizontal segment is colored blue,
    \item every vertical segment on the line $x = k$ is colored green,
    \item every vertical segment with an endpoint on the line $x+y = k$ is colored green, and
    \item if $s_1$ and $s_2$ are vertical segments such that $s_1$ is immediately above $s_2$ and $s_2$ is colored blue, then $s_1$ must also be colored blue.
\end{itemize}
\end{definition}

Let $\WT(k)$ denote the set of waterfalls of size $k$. From our discussion above, we have
\[
|\WT(k)| = \frac{1}{k} \binom{3k-2}{k-1}.
\]

We obtain the following curious combinatorial formula.

\begin{proposition}\label{prop:curious}
Let $k$ be a positive integer. For each Dyck path $D \in \Dyck(k)$, let us define the {\em weight} of $D$ to be
\[
\wt(D) := \prod_{i=1}^{k-1} \# \left\{ j \in \mathbb{Z} \, | \, i+j > k \text{ and } (i,j) \in D \right\}.
\]
Then,
\[
\sum_{D \in \Dyck(k)} \wt(D) = \frac{1}{k} \binom{3k-2}{k-1}.
\]
\end{proposition}

\begin{proof}
Note that the weight $\wt(D)$ is the number of waterfalls whose underlying Dyck paths are $D$. Therefore, $\sum_{D \in \Dyck(k)} \wt(D)$ counts the total number of waterfalls in $\WT(k)$.
\end{proof}

\subsection{A precise asymptotic formula for the number of rook placements for functions in $\widetilde{\mathcal{P}}_k$}\label{ss:precise-asymp} The goal of this subsection is to compute a precise asymptotic formula for $\#\RP(\lambda(f,N))$, for each $f \in \widetilde{\mathcal{P}}_k$, of the form
\[
\log\left( \#\RP(\lambda(f,N)) \right) = A_f \cdot N \log N + B_f \cdot N + C_f \cdot \log N + D_f + O_f(1/N),
\]
for positive integers $N \in k \mathbb{Z} := \{k, 2k, 3k, \ldots\}$. Since $\widetilde{\mathcal{P}}_k$ is a subclass of $\mathcal{P}$, the quantities $A_f$, $B_f$, and $C_f$ are the same as before.

In this subsection, we redefine our notations $\mu_i$ and $\beta_i$. These notations now have slightly different meanings from what they meant in Subsection~\ref{ss:asymp-P}. For $f \in \widetilde{\mathcal{P}}_k$ and for any $i \in [k]$, we know that $f$ is a linear function on the half-open interval $( (i-1)/k, i/k ]$. Let $\mu_i$ and $\beta_i$ be such that for $x \in ((i-1)/k, i/k]$, we have $f(x) = \mu_i x + \beta_i$.

Because $f \in \widetilde{\mathcal{P}}_k$, the discrepancy from rounding, $R(f,N)$ (as in Proposition~\ref{prop:RfN-is-Of1}), is zero. Furthermore, the effect from jumps (as in Equation~(\ref{eq:effect-from-jumps})) is also zero. Therefore, for $f \in \widetilde{\mathcal{P}}_k$ and for any positive integer $N \in k \mathbb{Z}$, we have
\begin{equation}
    \log \left( \# \RP(\lambda(f,N)) \right) = \sum_{i=1}^k \sum_{\frac{i-1}{k} N < n \le \frac{i}{k} N} \log((\mu_i+1)n + (\beta_i-1)N).
\end{equation}
Once again, we break the outer summation on the right-hand side above into when $i = 1$ and when $i \ge 2$. When $i = 1$, we have, by Stirling's formula,
\begin{align}
&\sum_{0 < n \le \frac{N}{k}} \log((\mu_i+1)n + (\beta_i-1)N) = \sum_{0 < n \le \frac{N}{k}} \log(n) = \log((N/k)!) \\
&= \frac{1}{k} \cdot N \log N - \frac{(\log k)+1}{k} \cdot N + \frac{1}{2} \log N + \frac{1}{2} \log \left( \frac{2\pi}{k} \right) + O(k/N),\label{eq:Stirling-wt-P}
\end{align}
for positive integers $N \in k \mathbb{Z}$.

When $2 \le i \le k$, we use the Euler-Maclaurin summation formula (cf. \cite[Appendix~B]{MV07}) to obtain
\begin{align}
    &\sum_{\frac{i-1}{k} N < n \le \frac{i}{k} N} \log((\mu_i+1)n + (\beta_i-1)N) \\
    &= \frac{1}{k} \cdot N \log N + \left\{ \int_{(i-1)/k}^{i/k} \log(f(x)+x-1) \, dx \right\} \cdot N \label{eq:EM-wt-P} \\
    &\hphantom{=} + \frac{1}{2} \cdot \log \left( \frac{(\mu_i+1) \cdot \frac{i}{k} + (\beta_i - 1)}{(\mu_i+1) \cdot \frac{i-1}{k} + (\beta_i - 1)} \right) + O_f(1/N).
\end{align}
Combining these terms, we obtain the following theorem.

\begin{theorem}\label{thm:asymp-wt-P}
Let $k$ be a positive integer. Let $f \in \widetilde{\mathcal{P}}_k$. We have
\[
\log \left( \# \RP(\lambda(f,N)) \right) = N \log N + B_f \cdot N + \frac{1}{2} \log N + D_f + O_f(1/N),
\]
for positive integers $N \in k \mathbb{Z}$, where $B_f$ is as given in Theorem~\ref{thm:asymp-P}, and
\[
D_f := \frac{1}{2} \log(2\pi) + \frac{1}{2} \int_0^1 \frac{xf'(x) - f(x) + 1}{x(f(x)+x-1)} dx.
\]
\end{theorem}

Note that one has to be careful about the integral in the formula of $D_f$. Since in general $f$ has a number of non-differentiable points, the derivative $f'$ might be undefined for some values of $x$. By the integral as expressed, we mean
\[
\int_0^1 \frac{xf'(x) - f(x) + 1}{x(f(x)+x-1)} dx = \sum_{i=1}^k \int_{(i-1)/k}^{i/k} \frac{xf'(x) - f(x) + 1}{x(f(x)+x-1)} dx.
\]
Since $f$ is linear in $((i-1)/k, i/k)$, the sum of integrals on the right-hand side is well-defined.

An illustration of Theorem~\ref{thm:asymp-wt-P} is given in Figure~\ref{fig:P3}. In the following examples, we compute explicit asymptotic formulas for certain functions.

\begin{example}\label{ex:B-low}
Suppose that $k$ is a positive integer. Let $f:[0,1] \to [0,1]$ be given as
\[
f(x) := \min\left\{ 1 + \frac{1}{k} - x, 1 \right\},
\]
for all $x \in [0,1]$. Note that $f \in \widetilde{\mathcal{P}}_k$. We have $B_f = - \log k - \frac{1}{k}$ and $D_f = \frac{1}{2} \log(2\pi/k)$. Therefore,
\[
\#\RP(\lambda(f,N)) \sim \sqrt{\frac{2\pi}{k}} \cdot N^{N+\frac{1}{2}} \cdot (k^{-1}e^{-1/k})^N,
\]
as $N \to \infty$, $N \in k\mathbb{Z}$.
\end{example}

\begin{example}\label{ex:D-high}
Suppose that $k$ is a positive integer. Let $f:[0,1] \to [0,1]$ be given as
\[
f(x) := \min \left\{ 1 + \frac{2-\left\lceil kx \right\rceil}{k}, 1 \right\},
\]
for all $x \in [0,1]$. Note that $f \in \widetilde{\mathcal{P}}_k$. We have $B_f = -\log k + \left( 2 - \frac{2}{k} \right) \log 2 - 1$ and $D_k = -\frac{1}{2}\log k + \frac{k}{2} \log 2 + \frac{1}{2} \log \pi$. Therefore,
\[
\#\RP(\lambda(f,N)) \sim \sqrt{\frac{2^k \pi}{k}} \cdot N^{N+\frac{1}{2}} \cdot \left( 2^{2-2/k} \, k^{-1} \, e^{-1} \right)^N,
\]
as $N \to \infty$, $N \in k \mathbb{Z}$.
\end{example}

As another application of our result, we can detect the number of {\em ground bumps} of Dyck paths analytically. If $D \in \Dyck(k)$ is a Dyck path from $(0,k)$ to $(k,0)$, then a {\em ground bump} of $\gamma$ is an intersection between $\gamma$ and the open line segment from $(0,k)$ to $(k,0)$. Recall that a partition $\lambda \in \cD_k$ can be thought of as a Dyck path from $(0,k)$ to $(k,0)$. In the following proposition, a {\em ground bump} of $\lambda$ is defined as a ground bump of the Dyck path corresponding to $\lambda$.

\begin{proposition}\label{prop:ground-bump-detect}
Let $\lambda$ be any nonempty partition such that $\RP(\lambda) \neq \varnothing$. Then, there exist positive real numbers $\alpha_1, \alpha_2, \alpha_3, \alpha_4 > 0$ such that
\[
\#\RP(N \odot \lambda) \sim N^{\alpha_1 N + \alpha_2} \cdot \alpha_3^N \cdot \alpha_4,
\]
as $N \to \infty$. Furthermore, $\alpha_1 = \lambda_1$ and
\[
\alpha_2 = \frac{(\#\text{ground bumps of } \lambda) + 1}{2}.
\]
\end{proposition}
\begin{proof}
Since $\RP(\lambda) \neq \varnothing$, we have that $\lambda \in \cD_{\lambda_1}$. As a Dyck path, $\lambda$ can be uniquely written as a concatenation $\lambda = \lambda^{(1)} \ast \lambda^{(2)} \ast \cdots \ast \lambda^{(p)}$ such that each $\lambda^{(i)}$ is a Dyck path without ground bumps. Observe that
\begin{equation}\label{eq:lambda-product}
    \#\RP(N \odot \lambda) = \prod_{i=1}^p \#\RP(N \odot \lambda^{(i)}),
\end{equation}
and that
\begin{equation}\label{eq:lambda-sum}
    \lambda_1 = \sum_{i=1}^p \lambda^{(i)}_1,
\end{equation}
where $\lambda^{(i)}_1$ denotes the first part of the partition $\lambda^{(i)}$.

For each $i \in [p]$, there is a unique corresponding function $f^{(i)} \in \widetilde{\mathcal{P}}_{\lambda_1^{(i)}}$. We have that for any positive integer $N$,
\[
\#\RP(N \odot \lambda^{(i)}) = \#\RP(\lambda(f^{(i)}, N \lambda^{(i)}_1)).
\]
(Note that the notation $\lambda$ on the right-hand side of the equation above is an operator, not a partition.) Now, Theorem~\ref{thm:asymp-wt-P} gives
\begin{align}
    \log(\#\RP(N \odot \lambda_i)) &= \lambda_1^{(i)} N \log N + (\lambda^{(i)}_1 \log \lambda^{(i)}_1) N + B_{f^{(i)}}\lambda^{(i)}_1 N + \frac{1}{2} \log N \label{eq:N-odot-lambda-i} \\
    &\hphantom{=} + \frac{1}{2} \log \lambda^{(i)}_1 + D_{f^{(i)}} + O_{\lambda}(1/N).
\end{align}
Combining Equations~(\ref{eq:lambda-product}), (\ref{eq:lambda-sum}), and (\ref{eq:N-odot-lambda-i}), we obtain
\begin{align}
    \log(\#\RP(N \odot \lambda)) &= \lambda_1 N \log N + \sum_{i=1}^p \left(\lambda^{(i)}_1 \log \lambda^{(i)}_1 + B_{f^{(i)}} \lambda^{(i)}_i \right) \cdot N \\
    &\hphantom{=} + \frac{p}{2} \log N + \sum_{i=1}^p \left( \frac{1}{2} \log \lambda^{(i)}_1 + D_{f^{(i)}} \right) + O_{\lambda}(1/N),
\end{align}
for positive integers $N$. Since $p-1$ is the number of ground bumps of $\lambda$, we have finished the proof.
\end{proof}

\begin{figure}
\begin{center}
\def\Pttfn[#1,#2,#3]{
\draw (0,0) grid (3,3); \draw[dashed] (0,3) -- (3,0); \draw[line width = 1 pt, color = red] (0,3) -- (1,3) -- (2,#1); \draw[line width = 1 pt, color = red] (2,#2) -- (3,#3); }
\def\h{5}
    \begin{tikzpicture}[scale = 0.3]
    \begin{scope}[shift={(0,6*\h)}]    
    \Pttfn[3,3,3]
    \node[anchor = west] at (4,1.5) {$\sim \sqrt{2\pi} \cdot N^{N+1/2} \cdot (e^{-1})^N$};
    \end{scope}
    
    \begin{scope}[shift={(0,5*\h)}]    
    \Pttfn[3,3,2]
    \node[anchor = west] at (4,1.5) {$\sim \sqrt{\frac{4}{3}\pi} \cdot N^{N+1/2} \cdot \left( 2^1 \, 3^{-1} \, e^{-2/3} \right)^N$};
    \end{scope}
    
    \begin{scope}[shift={(0,4*\h)}]    
    \Pttfn[3,3,1]
    \node[anchor = west] at (4,1.5) {$\sim \sqrt{\frac{2}{3} \pi} \cdot N^{N+1/2} \cdot \left( 2^{4/3} \, 3^{-1} \, e^{-1} \right)^N$};
    \end{scope}
    
    \begin{scope}[shift={(0,3*\h)}]    
    \Pttfn[3,2,2]
    \node[anchor = west] at (4,1.5) {$\sim \sqrt{\frac{8}{3} \pi} \cdot N^{N+1/2} \cdot \left( 2^{4/3} \, 3^{-1} \, e^{-1} \right)^N$};
    \end{scope}
    
    \begin{scope}[shift={(0,2*\h)}]    
    \Pttfn[3,2,1]
    \node[anchor = west] at (4,1.5) {$\sim \sqrt{\frac{4}{3} \pi} \cdot N^{N+1/2} \cdot \left( 2^{2/3} \, 3^{-1} \, e^{-2/3} \right)^N$};
    \end{scope}
    
    \begin{scope}[shift={(0,\h)}]    
    \Pttfn[2,2,2]
    \node[anchor = west] at (4,1.5) {$\sim \sqrt{\frac{4}{3} \pi} \cdot N^{N+1/2} \cdot \left( 2^{2/3} \, 3^{-1} \, e^{-2/3} \right)^N$};
    \end{scope}
    
    \begin{scope}[shift={(0,0)}]    
    \Pttfn[2,2,1]
    \node[anchor = west] at (4,1.5) {$\sim \sqrt{\frac{2}{3} \pi} \cdot N^{N+1/2} \cdot \left( 3^{-1} \, e^{-1/3} \right)^N$};
    \end{scope}
    \end{tikzpicture}
\end{center}
\caption{The seven functions in the class $\widetilde{\mathcal{P}}_3$ and their corresponding asymptotic formulas for $\#\operatorname{RP}(\lambda(f,N))$, as $N \to \infty$, $N \in 3\mathbb{Z}$.\label{fig:P3}}
\end{figure}

\subsection{Properties of $D_f$}\label{ss:prop-D} Theorem~\ref{thm:asymp-wt-P} gives an integral formula for $D_f$. In applications, it is also useful to have the following formula, which is immediate from Equations~(\ref{eq:Stirling-wt-P}) and (\ref{eq:EM-wt-P}).

\begin{proposition}\label{prop:D_f-formula}
For any function $f \in \widetilde{\mathcal{P}}$, we have
\[
D_f = \frac{1}{2} \log(2\pi f(1)) + \frac{1}{2} \sum_{a \in (0,1)} \log \left( \frac{f(a) + a - 1}{\lim_{x \searrow a} f(x) + x - 1} \right).
\]
Note that the sum on the right-hand side is finite, since there are only finitely many discontinuous points for $f$.
\end{proposition}

Let $p \ge 1$ be any positive real number. Let $(\widetilde{\mathcal{P}}, L^p)$ be the metric space obtained from endowing $\widetilde{\mathcal{P}}$ with the $L^p$ norm. The map $f \mapsto D_f$ induces a well-defined map
\[
\boldsymbol{D}: (\widetilde{\mathcal{P}}, L^p) \to \mathbb{R}.
\]
In Proposition~\ref{prop:B-discontinuous}, we have seen that the map $\boldsymbol{B}:(\mathcal{P}, L^p) \to (\mathbb{R}, \operatorname{Euclid})$ is discontinuous everywhere. The following proposition says that $\boldsymbol{D}$, considered as a function from $(\widetilde{\mathcal{P}}, L^p)$ to $(\mathbb{R}, \operatorname{Euclid})$, exhibits a similar topological property.

\begin{proposition}
The map $\boldsymbol{D}:(\widetilde{\mathcal{P}}, L^p) \to (\mathbb{R}, \operatorname{Euclid})$ is discontinuous everywhere on $\widetilde{\mathcal{P}}$.
\end{proposition}
\begin{proof}
Let $f \in \widetilde{\mathcal{P}}$ be arbitrary. Let $k$ be a positive integer for which $f \in \widetilde{\mathcal{P}}_k$. For each positive integer $n \ge k+1$, define
\[
g_n(x) = \begin{cases}
f(x) & \text{if } x \le 1 - \frac{1}{n}, \\
n^{-1} + n^{-2} & \text{if } x > 1 - \frac{1}{n}.
\end{cases}
\]
Note that $\{g_n\}$ is a sequence of functions in $\widetilde{\mathcal{P}}$ that converges in $L^p$ to $f$. On the other hand, we obtain from Proposition~\ref{prop:D_f-formula} that
\[
D_{g_n} = \frac{1}{2} \log(n+1) + D_f + \frac{1}{2} \log \left( \frac{f(1-\frac{1}{n}) - \frac{1}{n}}{f(1)} \right).
\]
Since $\lim_{x \nearrow 1} f(x) = f(1)$, the third term on the right-hand side converges to $0$ as $n \to \infty$. The second term does not depend on $n$. The first term goes to $+\infty$, as $n \to \infty$. Therefore, $D_{g_n} \to \infty$, as $n \to \infty$.
\end{proof}

\subsection{Bounds for $B_f$ and $D_f$}\label{ss:bounds-B-D} In this subsection, we determine the extremal values for both $B_f$ and $D_f$ among all functions $f \in \widetilde{\mathcal{P}}_k$.

\begin{proposition}\label{prop:range-B}
Let $k$ be a positive integer. Let $f \in \widetilde{\mathcal{P}}_k$. Then,
\[
-\log k - \frac{1}{k} \le B_f \le -1.
\]
Both bounds are tight. The lower bound is attained if and only if $f$ is the function in Example~\ref{ex:B-low}. The upper bound is attained if and only if $f(x) = 1$ for every $x \in [0,1]$.
\end{proposition}
\begin{proof}
The upper bound follows from Proposition~\ref{prop:from-int-form-B}(a). For the equality case of the upper bound, note that since $\lim_{x \nearrow 1} f(x) = f(1)$, if $f(x) = 1$ for all $x \in [0,1)$, then $f(1)$ must also be $1$.

For the lower bound, note that for any function $f \in \widetilde{\mathcal{P}}_k$, we have $\loft(f) \ge 1/k$. Therefore, by Proposition~\ref{prop:basic-loft}, we have
\[
B_f = \int_0^1 \log(f(x)+x-1) \, dx \ge \int_0^{1/k} \log(x) \, dx + \int_{1/k}^1 \log(1/k) \, dx
\]
Computing the last expression yields the desired lower bound. The equality case happens when $f(x) = 1$ on $[0,1/k]$ and $f(x) + x - 1 = 1/k$ on $(1/k,1]$. This is exactly if and only if $f$ is the function in Example~\ref{ex:B-low}.
\end{proof}

\begin{proposition}\label{prop:range-D}
Let $k$ be a positive integer. Let $f \in \widetilde{\mathcal{P}}_k$. Then,
\[
\frac{1}{2} \log \left( \frac{2\pi}{k} \right) \le D_f \le \frac{1}{2} \log \left( \frac{2^k \pi}{k} \right).
\]
Both bounds are tight. The lower bound is attained if and only if $f$ is continuous and $f(1) = 1/k$ (i.e., if and only if $f$ is in the Catalan-numerous family discussed in Remark~\ref{rmk:f(1)-is-1/k-family}). The upper bound is attained if and only if $f$ is the function in Example~\ref{ex:D-high}.
\end{proposition}
\begin{proof}
For the lower bound, note that by the formula in Proposition~\ref{prop:D_f-formula}, it is immediate that
\[
D_f \ge \frac{1}{2} \log(2\pi f(1)).
\]
Since $f \in \widetilde{\mathcal{P}}_k$, we have $f(1) \ge 1/k$. Combining the two inequalities yields the desired lower bound. The equality is attained if and only if $f(1) = 1/k$ and there are no ``jumps." In other words, $f$ is continuous and $f(1) = 1/k$.

For the upper bound, we use the following strategy. We start with an arbitrary function $f \in \widetilde{\mathcal{P}}_k$, and then we keep transforming the function (if possible) in a number of steps so that in each step $D_f$ becomes larger. We claim that we can always end at the unique extremal function in Example~\ref{ex:D-high}.

First, start with any function $f \in \widetilde{\mathcal{P}}_k$. Consider whether $f$ has a linear piece with a strictly negative slope. If so -- say over $((i-1)/k,i/k]$, $f$ has a negative slope -- modify the function $f$ so that over $((i-1)/k,i/k]$, it becomes constant with the value $\lim_{x \searrow (i-1)/k} f(x)$ instead. The new function remains in $\widetilde{\mathcal{P}}_k$, and the value $D_f$ strictly increases.

Second, now assume that the function $f$ is already piecewise constant. Consider the value of $f$ over $(1-1/k,1]$. If it is strictly greater than $2/k$, change the value to $2/k$. This change strictly increases $D_f$. Then, consider the value of $f$ over $(1-2/k, 1-1/k]$. If it is strictly greater than $3/k$, change the value to $3/k$. Keep going in this manner from the right to the left. The resulting function is the unique function in Example~\ref{ex:D-high}. This proves the upper bound.

Note that in each step, if a change is made, the value of $D_f$ increases {\em strictly}. This shows that the equality case for the upper bound happens if and only if $f$ is the unique function in Example~\ref{ex:D-high}.
\end{proof}

\bigskip

\section{Cumulative X-rays of rook placements}\label{s:X-rays}
\subsection{Marginal Probabilities}
In the following, for a finite nonempty set $S$, we denote by $\Unif(S)$ the uniform distribution on $S$. The notation $X \sim \Unif(S)$ means that $X$ is a uniform random variable so that $\forall s \in S, \bbP(X = s) = |S|^{-1}$.

Let $n$ be a positive integer. Let $\lambda \in \cD_n$ be a partition. In what follows, let us consider our partitions in the French notation so that the boxes of $\lambda$ are bottom- and left-aligned and there are $\lambda_1$ boxes on the bottom row, $\lambda_2$ boxes on the second row from the bottom, and so on.

\begin{proposition}\label{prop:margin-prob}
Let $\pi \sim \Unif(\RP(\lambda))$. Let $i, j \in [n]$.
\begin{itemize}
    \item[(a)] If $\lambda_{n+1-i} < j$, then $\bbE\!\left( \pi_{ij} \right) = 0$.
    \item[(b)] If $\lambda_{n+1-i} \ge j$, and suppose $i' \in [n]$ is the smallest index such that $\lambda_{n+1-i'} \ge j$, then
\[
\bbE\!\left( \pi_{ij} \right) = \left( \prod_{i' \le t < i} \frac{\lambda_{n+1-t} - t}{\lambda_{n+1-t} - t + 1} \right) \cdot \frac{1}{\lambda_{n+1-i} - i + 1}.
\]
\end{itemize}
\end{proposition}
\begin{proof}
\textbf{(a)} It follows immediately from the definition of $\RP(\lambda)$ (cf. the beginning of Section~\ref{s:rook-placements}) that $\pi_{ij} = 0$.

\textbf{(b)} Suppose that $\mu$ is the partition obtained by removing the $i^{\text{th}}$ row from the top (the row corresponding to $\lambda_{n+1-i}$) and the $j^{\text{th}}$ column from the left from the Young diagram of $\lambda$. Observe that the probability that $\pi_{ij} = 1$ is $\#\RP(\mu)/\#\RP(\lambda)$.

The formula in Equation~(\ref{eq:RP-prod}) gives
\[
\#\RP(\mu) = \left( \prod_{t < i' \text{ or } t > i} (\lambda_{n+1-t} - (t-1))\right) \left( \prod_{i' \le t < i} (\lambda_{n+1-t} - t)\right).
\]
Therefore,
\[
\bbE\!\left( \pi_{ij} \right) = \bbP(\pi_{ij} = 1 ) = \frac{\#\RP(\mu)}{\#\RP(\lambda)} = \left( \prod_{i' \le t < i} \frac{\lambda_{n+1-t} - t}{\lambda_{n+1-t} - t + 1} \right) \cdot \frac{1}{\lambda_{n+1-i} - i + 1},
\]
as desired.
\end{proof}

The following corollary is immediate from Proposition~\ref{prop:margin-prob}.

\begin{corollary}\label{cor:marginal-bound}
If $\pi \sim \Unif(\RP(\lambda))$ and $i, j \in [n]$, then
\[
\bbE\!\left( \pi_{ij} \right) \le \frac{1}{\lambda_{n+1-i} - i + 1}.
\]
\end{corollary}

We also have the following result. Let $\lambda' \in \cD_n$ denote the conjugate partition of $\lambda$.
\begin{corollary}\label{cor:i1i2j1j2}
Let $\pi \sim \Unif(\RP(\lambda))$. Suppose that $i_1, i_2, j_1, j_2 \in [n]$ satisfy $\lambda_{n+1-i_1} \ge j_1$ and $\lambda_{n+1-i_2} \ge j_1$. If $\lambda_{i_1} = \lambda_{i_2}$ and $\lambda'_{j_1} = \lambda'_{j_2}$, then $\bbE\!\left( \pi_{i_1 j_1} \right) = \bbE\!\left( \pi_{i_2 j_2} \right)$.
\end{corollary}
\begin{proof}
We proceed in a similar manner to how we proved Proposition~\ref{prop:margin-prob}(b). Namely, let $\mu^{(1)}$ (and $\mu^{(2)}$) denote the resulting partition from removing the $(i_1, j_1)$ (and resp. $(i_2, j_2)$) box (together with the row and the column) from $\lambda$. It is not hard to see that $\mu^{(1)} = \mu^{(2)}$. This finishes the proof.
\end{proof}

Like before, we may think of $\lambda$ as a Dyck path from $(0,n)$ to $(n,0)$, which we can write as the following concatenation
\[
\lambda = R^{r_1} D^{d_1} R^{r_2} D^{d_2} \cdots,
\]
where $R$ denotes a unit step to the right, and $D$ denotes a unit step down. The equation above means that $\lambda$ starts by going $r_1$ steps to the right, and then $d_1$ steps down, and so on. Let us refer to the quantity $\min\{r_1, d_1, r_2, d_2, \ldots\}$ as the {\em minimum run} of $\lambda$, denoted $\mr(\lambda)$. For instance, since $\lambda \in \cD_n$, we have $\mr(\lambda) \le n$, where the equality is attained if and only if $\lambda = R^n D^n$.

\begin{proposition}
Let $\pi \sim \Unif(\RP(\lambda))$. Let $t \le \mr(\lambda)$ be a positive integer. Then, for any $t$ different boxes $b_1 = (i_1, j_1)$, $b_2 = (i_2, j_2)$, $\ldots$, $b_t = (i_t, j_t) \in [n]^2$, the probability that $\pi$ is $1$ in all these $t$ boxes is
\[
\bbE\!\left( \pi_{b_1} \pi_{b_2} \cdots \pi_{b_t} \right) \le \frac{1}{\mr(\lambda) \cdot (\mr(\lambda) - 1) \cdots (\mr(\lambda) - t + 1)}.
\]
\end{proposition}
\begin{proof}
If any of the $t$ boxes is ``outside" the Young diagram of $\lambda$, we are done. Suppose that all these boxes are inside the Young diagram (i.e., $\lambda_{n+1-i_1} \ge j_1$ and so on). Since removing a box (together with its row and its column) reduces the minimum run by at most $1$, it suffices to show that if we remove one box $b$ (together with its row and its column) from $\lambda$ and obtain a new partition $\mu$, then
\begin{equation}\label{eq:mu-lambda-le-mr}
    \frac{\#\RP(\mu)}{\#\RP(\lambda)} \le \frac{1}{\mr(\lambda)}.
\end{equation}
Consider a box $b = (i,j)$ such that $\lambda_{n+1-i} \ge j$. Let $i'$ denote the smallest index for which $\lambda_{n+1-i} = \lambda_{n+1-i'}$. Let $i''$ denote the largest index for which $\lambda_{n+1-i} = \lambda_{n+1-i''}$. By the definition of the minimum run, we have
\begin{equation}\label{eq:i''i'}
    i'' - i' + 1 \ge \mr(\lambda).
\end{equation}
By Corollaries~\ref{cor:marginal-bound} and \ref{cor:i1i2j1j2}, we have
\begin{equation}
    \bbE\!\left( \pi_{ij} \right) = \bbE\!\left( \pi_{i'j} \right) \le \frac{1}{\lambda_{n+1-i'} - i' + 1} = \frac{1}{(\lambda_{n+1-i''} - i'') + i'' - i' + 1} \overset{(\ref{eq:i''i'})}{\le} \frac{1}{\mr(\lambda)},
\end{equation}
which implies (\ref{eq:mu-lambda-le-mr}).
\end{proof}

Since the minimum run also grows as we dilate partitions, we immediately have the following corollary.

\begin{corollary}
Let $t$ and $N$ be positive integers such that $t \le N$. Suppose that $\pi \sim \Unif(\RP(N \odot \lambda))$. Then, for any $t$ different boxes $b_1, b_2, \ldots, b_t \in [nN]^2$, we have
\[
\bbE\!\left( \pi_{b_1} \pi_{b_2} \cdots \pi_{b_t} \right) \le \frac{1}{N(N-1) \cdots (N-t+1)}.
\]
\end{corollary}

The proposition below shows that these marginal probabilities $\bbP(\pi_{i,j} = 1)$ behave nicely in the following sense, when we dilate partitions.

\begin{proposition}\label{prop:margin-scale}
If $\pi \sim \Unif(\RP(N \odot \lambda))$ and $\pi^{\downarrow} \sim \Unif(\RP(\lambda))$. Then, for any $i, j \in [nN]$, we have
\[
\bbE\!\left( \pi_{i,j} \right) = \frac{1}{N} \cdot \bbE \! \left( \pi^{\downarrow}_{\left\lceil i/N \right\rceil, \left\lceil j/N \right\rceil}\right)
\]
\end{proposition}
\begin{proof}
This follows from a direct computation using Proposition~\ref{prop:margin-prob}(b).
\end{proof}

\subsection{Cumulative X-rays}\label{ss:cumulative-X-rays}
Let $n$ be a positive integer. Suppose that a permutation $\pi \in S_n$ is given. The {\em cumulative X-ray} of $\pi$ is the piecewise constant function $\xi_{\pi}:[0,2n] \to \mathbb{R}$ given by
\[
\xi_{\pi}(t) := \sum_{\substack{i,j \in [n] \\ i+j \le t}} \pi_{ij}.
\]
We also define the normalized version of cumulative X-rays. The {\em normalized cumulative X-ray} of $\pi \in S_n$ is the piecewise constant function $\widetilde{\xi}_{\pi}:[0,2] \to [0,1]$ given by
\[
\widetilde{\xi}_{\pi}(t) := \frac{1}{n} \cdot \xi_{\pi}(nt).
\]

The following is a counting lemma which is easy to prove.

\begin{lemma}\label{lemma:counting}
For each real number $\phi \in \mathbb{R}$ and each positive integer $N \in \mathbb{Z}_{\ge 1}$, let
\[
S(\phi;N) := \#\left\{ (x,y) \in [N] \times [N] \, | \, x+y \le \phi N \right\}.
\]
Then, we have the following.
\begin{itemize}
    \item[(a)] If $\phi \le 0$, then $S(\phi;N) = 0$.
    \item[(b)] If $0 \le \phi \le 1$, then
    \[
    S(\phi;N) = \binom{\left\lfloor \phi N \right\rfloor}{2} = \frac{\phi^2}{2} N^2 + O(\phi N).
    \]
    \item[(c)] If $1 \le \phi \le 2$, then
    \[
    S(\phi;N) = N^2 - \binom{2N-\left\lfloor \phi N \right\rfloor + 1}{2} = \left( - \frac{\phi^2}{2} + 2 \phi - 1 \right) N^2 + O(\phi N).
    \]
    \item[(d)] If $\phi \ge 2$, then $S(\phi;N) = N^2$.
\end{itemize}
In (b) and (c), the implicit constants are absolute.
\end{lemma}

For convenience, let us reserve the symbol $\fc$. In the following, we let $\fc$ denote the function $\fc: \mathbb{R} \to \mathbb{R}$ given by
\[
\fc(t) :=
\begin{cases}
0 & \text{if } t \le 0, \\
\frac{t^2}{2} & \text{if } 0 < t \le 1, \\
-\frac{t^2}{2} + 2t - 1 & \text{if } 1 < t \le 2, \text{ and} \\
1 & \text{if } t \ge 2.
\end{cases}
\]
Thus, Lemma~\ref{lemma:counting} says that
\begin{equation}\label{eq:counting-s-c}
    S(\phi;N) = \fc(\phi)N^2 + O(N),
\end{equation}
for positive integers $N$.

For each partition $\lambda \in \cD_n$, let us define a function $\fm_{\lambda} : [0,2] \to [0,1]$ by
\begin{equation}\label{eq:limit-shape}
    \fm_{\lambda}(t) := \frac{1}{n} \underset{\substack{x,y \in (0,n] \\ x+y \le nt }}{\int\int} \bbE\!\left( \pi_{\left\lceil x \right\rceil, \left\lceil y \right\rceil}\right) \, dx \, dy,
\end{equation}
where $\pi \sim \Unif(\RP(\lambda))$.

\begin{proposition}\label{prop:conv-of-means}
Let $\lambda \in \cD_n$ and $N \in \mathbb{Z}_{\ge 1}$. Suppose that $\pi \sim \Unif(\RP(N \odot \lambda))$. Then, for any real number $t \in [0,2]$, we have
\[
\bbE\!\left( \widetilde{\xi}_{\pi}(t) \right) = \fm_{\lambda}(t) + O_{\lambda}\!\left( \frac{1}{N} \right).
\]
\end{proposition}
\begin{proof}
From Proposition~\ref{prop:margin-scale}, we have
\begin{align}
    \bbE\!\left( \widetilde{\xi}_{\pi}(t) \right) &= \frac{1}{nN^2} \sum_{\substack{i,j \in [nN] \\ i+j \le nNt}} \bbE \! \left( \pi^{\downarrow}_{\left\lceil i/N \right\rceil, \left\lceil j/N \right\rceil}\right) \\
    &= \frac{1}{nN^2} \sum_{a=1}^n \sum_{b=1}^n \sum_{\substack{i,j \in [nN] \\ i+j \le nN t \\ (\left\lceil i/N \right\rceil, \left\lceil j/N \right\rceil) = (a,b)}} \bbE \! \left( \pi^{\downarrow}_{a,b}\right). \label{eq:triple-a-b}
\end{align}

Note that the innermost summation in the last expression above is over pairs $(i,j)$ of positive integers such that (i) $i,j \le nN$, (ii) $i+j \le nNt$, (iii) $aN - N +1 \le i \le aN$, and (iv) $bN - N + 1 \le j \le bN$. By translation, the number of such pairs is exactly the number of $(x,y) \in [N]^2$ such that $x+y \le (nt-a-b+2)N$. By our discussion above, the number is exactly $S(nt-a-b+2;N)$.

Therefore, Equation~(\ref{eq:counting-s-c}) gives
\begin{align}
    \bbE\!\left( \widetilde{\xi}_{\pi}(t) \right) &= \frac{1}{n} \sum_{a = 1}^n \sum_{b = 1}^n \bbE \! \left( \pi^{\downarrow}_{a,b}\right) \left( \fc(nt-a-b+2) + O\!\left( \frac{1}{N} \right)\right) \\
    &= \frac{1}{n} \sum_{a = 1}^n \sum_{b = 1}^n \bbE \! \left( \pi^{\downarrow}_{a,b}\right)\fc(nt-a-b+2) + O_{\lambda}\!\left( \frac{1}{N} \right).
\end{align}
On the other hand, by the definition of $\fm_{\lambda}$, we have
\begin{align}
    \fm_{\lambda}(t) &= \frac{1}{n} \sum_{a = 1}^n \sum_{b = 1}^n \bbE\!\left( \pi^{\downarrow}_{a,b} \right) \int_{b-1}^b \int_{a-1}^a \boldsymbol{1}(x+y \le nt) \, dx \, dy \\
    &= \frac{1}{n} \sum_{a = 1}^n \sum_{b = 1}^n \bbE\!\left( \pi^{\downarrow}_{a,b} \right) \int_0^1 \int_0^1 \boldsymbol{1}(x+y \le nt-a-b+2) \, dx \, dy \label{eq:simple-change-a-b} \\
    &= \frac{1}{n} \sum_{a = 1}^n \sum_{b = 1}^n \bbE \! \left( \pi^{\downarrow}_{a,b}\right)\fc(nt-a-b+2), \label{eq:double-a-b}
\end{align}
where $\boldsymbol{1}$ denotes the indicator function, and Equation~(\ref{eq:simple-change-a-b}) follows from simple changes of variables $x \mapsto x+a-1$ and $y \mapsto y+b-1$. Combining Equations~(\ref{eq:triple-a-b}) and (\ref{eq:double-a-b}) finishes the proof.
\end{proof}

In particular, Proposition~\ref{prop:conv-of-means} implies a convergence of expectations. If for each $N$, we have a random variable $\pi^{(N)} \sim \Unif(\RP(N \odot \lambda))$, then for any fixed $t \in [0,2]$, the sequence $\left\{ \bbE\!\left( \widetilde{\xi}_{\pi^{(N)}}(t) \right) \right\}_{N = 1}^{\infty}$ converges to $\fm_{\lambda}(t)$.

\smallskip

The author of the present paper gives the following conjecture about this function $\fm_{\lambda}$.

\begin{conjecture}\label{conj}
Let $\lambda \in \cD_n$ be a fixed partition. Fix a positive real number $\varepsilon > 0$. Suppose that $\pi \sim \Unif(\RP(N \odot \lambda))$. Then,
\[
\bbP\!\left( \sup_{t \in [0,2]} \left| \widetilde{\xi}_{\pi}(t) - \fm_{\lambda}(t) \right| < \varepsilon \right) \to 1,
\]
as $N \to \infty$.
\end{conjecture}

In the next subsection, we give a proof of Conjecture~\ref{conj} in the special case of uniformly random permutations.

\subsection{Limit shape for normalized cumulative X-rays of random permutations}\label{ss:limit-random-permutation} Let $N$ be a positive integer, and let $\pi \sim \Unif(S_N)$ be a uniformly random permutation. For each $k \in [N+1]$, we let
\[
X_k := \xi_{\pi}(k).
\]
Let's also define, for each $k = 2, 3, \ldots, 2N$, the {\em $k^{\text{th}}$ X-ray component}
\[
x_k := \sum_{\substack{i,j \in [N] \\ i+j = k}} \pi_{ij}.
\]
Indeed, there is a simple relation between these notations: $X_k = x_2 + x_3 + \cdots + x_k$. 

The goal of this subsection is to prove the following result.

\begin{theorem}\label{thm:limit-shape}
Let $\varepsilon > 0$ be any fixed positive real number. Then,
\[
\bbP\!\left( \forall k \in [N+1], \left| X_k - \frac{k(k-1)}{2N} \right| < \varepsilon N \right) = 1 - O_{\varepsilon}\! \left( \frac{\log N}{N \, \log \log N} \right),
\]
as $N \to \infty$.
\end{theorem}

As a corollary of Theorem~\ref{thm:limit-shape}, we obtain a proof of Conjecture~\ref{conj} in the very special case when $\lambda = \square$ is a partition with one box. In this case, the function $\fm_{\lambda}$ coincides with $\fc|_{[0,2]}$. From the definition of $\fc$, we see that the graph of this function is a concatenation of two parabolas.

Here is our rough strategy for proving the theorem. First, we show that with high probability the largest X-ray component $\max_k x_k$ is small. Second, we give an upper bound on the size of the variance $\Var(X_k)$. Third, we argue that since the X-ray components are small with high probability, it suffices to establish the bound
\[
\left| X_k - \frac{k(k-1)}{2N} \right| < \varepsilon N
\]
with high probability for all $k$ simultaneously in a certain subset of $[N+1]$, instead of the whole $[N+1]$. Fourth, we use Chebyshev's tail bound to show that we have the bound
\[
\left| X_k - \frac{k(k-1)}{2N} \right| < \varepsilon N
\]
for all $k$ in the mentioned subset of $[N+1]$ simultaneously with high probability. This finishes the proof.

\smallskip

Now we begin the first step of our strategy.

\begin{proposition}
Let $t \le N$ be a positive integer. We have
\[
\sum_{k=2}^{N+1} \bbP\!\left( x_k \ge t \right) \le \frac{N+1}{(t+1)!}.
\]
\end{proposition}
\begin{proof}
The event $x_k \ge t$ is equivalent to the event that there exists a $t$-subset $\{i_1, i_2, \ldots, i_t\}$ of indices in $[k-1]$ such that
\[
\pi_{i_1, k-i_1} = \pi_{i_2, k-i_2} = \cdots = \pi_{i_t, k-i_t} = 1.
\]
Since
\begin{align}
    \bbP\!\left(\pi_{i_1, k-i_1} = \pi_{i_2, k-i_2} = \cdots = \pi_{i_t, k-i_t} = 1\right) &= \bbE\!\left(\pi_{i_1, k - i_1} \pi_{i_2, k - i_2} \cdots \pi_{i_t, k - i_t}\right) \\
    &= \frac{1}{N(N-1) \cdots (N-t+1)},
\end{align}
we have
\begin{equation}
    \bbP\!\left( x_k \ge t \right) \le \frac{1}{t!} \cdot \frac{(k-1)(k-2) \cdots (k-t)}{N(N-1) \cdots (N-t+1)}.
\end{equation}
Therefore, by telescoping, we obtain
\begin{align}
    \sum_{k=2}^{N+1} \bbP\!\left( x_k \ge t \right) &\le \sum_{k=2}^{N+1} \frac{(k-1)(k-2) \cdots (k-t)}{N(N-1) \cdots (N-t+1)} \\
    &= \frac{1}{t!} \sum_{k=2}^{N+1} \frac{k(k-1)\cdots(k-t) - (k-1)(k-2) \cdots (k-t-1)}{(t+1) \cdot N(N-1) \cdots (N-t+1)} \\
    &= \frac{N+1}{(t+1)!},
\end{align}
as desired.
\end{proof}

By using the bound
\begin{equation}
    \bbP\!\left( \max\{x_2, x_3, \ldots, x_{N+1}\} \ge t \right) \le \sum_{k=2}^{N+1} \bbP\!\left( x_k \ge t \right),
\end{equation}
together with Stirling's formula, we obtain the following corollary.

\begin{corollary}\label{cor:max-rare}
For all sufficiently large positive integers $N$, we have
\[
\bbP\!\left( \max\{x_2, x_3, \ldots, x_{N+1}\} \ge \frac{3 \log N}{\log \log N} \right) \le \frac{1}{N}.
\]
\end{corollary}

Next is the second step of the strategy. For each $2 \le k \le N+1$, it is easy to see that $\bbE\!\left( X_k \right) = \frac{k(k-1)}{2N}$.

\begin{proposition}\label{prop:E-X-squared}
Let $N \ge 2$ be a positive integer. For $2 \le k \le N+1$, we have
\[
\bbE\!\left( X_k^2 \right) = \frac{k(k-1)}{2N} + \frac{k(k-1)(k-2)(3k-5)}{12N(N-1)}.
\]
\end{proposition}
\begin{proof}
Observe that
\begin{equation}
    \bbE\!\left( X_k^2 \right) = \sum_{\ell = 2}^k \sum_{\substack{i, j \in [N] \\ i+j = \ell}} \sum_{\substack{i', j' \in [N] \\ i' + j' \le k}} \bbE\!\left( \pi_{ij} \pi_{i'j'} \right).
\end{equation}

For the innermost summation on the right-hand side above, there are three cases. First, if $(i',j') = (i,j)$, then $\bbE\!\left( \pi_{ij} \pi_{i'j'} \right) = 1/N$. Second, if $(i',j') \neq (i,j)$ but $(i',j')$ is on either the same row or the same column as $(i,j)$, then $\bbE\!\left( \pi_{ij} \pi_{i'j'} \right) = 0$. There are $2k-\ell-2$ such ordered pairs. For the rest, the expectation $\bbE\!\left( \pi_{ij} \pi_{i'j'} \right)$ is $\frac{1}{N(N-1)}$. Therefore,
\begin{equation}
    \bbE\!\left( X_k^2 \right) = \sum_{\ell = 2}^k \sum_{\substack{i, j \in [N] \\ i+j = \ell}} \left\{ \frac{1}{N} + \left( \frac{k(k-1)}{2} - 2k + \ell + 1 \right) \cdot \frac{1}{N(N-1)} \right\}.
\end{equation}
Simplify to finish.
\end{proof}

\begin{proposition}\label{prop:variance-bound}
Let $N \ge 2$ be a positive integer. For $2 \le k \le N+1$, we have
\[
\Var(X_k) < \frac{k^2}{2N}.
\]
\end{proposition}
\begin{proof}
From Proposition~\ref{prop:E-X-squared} and some algebraic manipulation, we obtain
\begin{equation}
    \Var(X_k) = \bbE\!\left(X_k^2\right) - \left(\bbE X_k\right)^2 = \frac{k(k-1)}{2N} - \frac{k(k-1)(4k-5)}{6N(N-1)} + \frac{k^2(k-1)^2}{4N^2(N-1)}.
\end{equation}
It is not hard to see that
\[
\frac{k(k-1)(4k-5)}{6N(N-1)} \ge \frac{k^2(k-1)^2}{4N^2(N-1)},
\]
whence $\Var(X_k) \le \frac{k(k-1)}{2N} < \frac{k^2}{2N}$.
\end{proof}

We now turn to the third step of the described strategy. For each real number $\varepsilon > 0$, we let $E_{\varepsilon}$ denote the event
\begin{equation}
    \forall k \in \{2, 3, \ldots, N+1\}, \left| X_k - \frac{k(k-1)}{2N} \right| < \varepsilon N.
\end{equation}
For each positive integer $g \le N+1$, we let $E_{g, \varepsilon}$ denote the event
\begin{equation}
    \forall k \in g \mathbb{Z} \cap [N+1], \left| X_k - \frac{k(k-1)}{2N} \right| < \varepsilon N.
\end{equation}
Here, $g \mathbb{Z} \cap [N+1]$ refers to the set $\left\{g, 2g, \ldots, \left\lfloor \frac{N+1}{g} \right\rfloor g\right\}$.

\begin{proposition}\label{prop:reduction}
Let $\varepsilon > 0$ be a positive real number. For all sufficiently large positive integers $N$, if
\[
g < \frac{\varepsilon}{8} \cdot \frac{N \, \log \log N}{\log N},
\]
then
\[
\bbP\!\left(E_{\varepsilon}\right) \ge \bbP\!\left( E_{g, \varepsilon/2} \right) - \frac{1}{N}.
\]
\end{proposition}
\begin{proof}
Let $A$ denote the event that $\max\{x_2, x_3, \ldots, x_{N+1}\} \ge \frac{3 \log N}{\log \log N}$. By Corollary~\ref{cor:max-rare}, it suffices to show that for all sufficiently large positive integers $N$, we have the inclusion $E_{g, \varepsilon} \subseteq E_{\varepsilon} \cup A$.

Consider any event in $E_{g, \varepsilon/2} \setminus A$. In this case, $\left| X_{k'} - \frac{k'(k'-1)}{2N} \right| < \frac{\varepsilon N}{2}$ for any $k'$ divisible by $g$, and $x_{\ell} < \frac{3 \log N}{\log \log N}$, for any $\ell$. We claim that for any $k \in \{2, 3, \ldots, N+1\}$, we have $\left| X_k - \frac{k(k-1)}{2N} \right| < \varepsilon N$.

Note that we can find an integer $k' \in \{2, 3, \ldots, N+1\}$ which is a multiple of $g$ such that $|k' - k| < g$. By the triangle inequality, we have
\begin{equation}
    \left| X_k - \frac{k(k-1)}{2N} \right| \le |X_k - X_{k'}| + \left| X_{k'} - \frac{k'(k'-1)}{2N} \right| + \left| \frac{k'(k'-1)}{2N} - \frac{k(k-1)}{2N}\right|.
\end{equation}
The first term on the right-hand side is
\begin{equation}
    |X_k - X_{k'}| \le g \cdot \frac{3 \log N}{\log \log N} < \frac{3}{8} \varepsilon N.
\end{equation}
The second term is less than $\frac{\varepsilon N}{2}$. The third term is
\begin{equation}
    \left| \frac{k'(k'-1)}{2N} - \frac{k(k-1)}{2N}\right| = \left| \frac{(k'-k)(k'+k-1)}{2N}\right| \le g \le \frac{1}{8} \varepsilon N,
\end{equation}
for all sufficiently large $N$. Therefore,
\[
\left| X_k - \frac{k(k-1)}{2N} \right| < \varepsilon N,
\]
for all sufficiently large positive integers $N$. This shows that $E_{g, \varepsilon/2} \setminus A \subseteq E_{\varepsilon}$, finishing the proof.
\end{proof}

We have arrived at the final step of our strategy.

\begin{proof}[Proof of Theorem~\ref{thm:limit-shape}]
Let $E^c_{g, \varepsilon/2}$ denote the complement of the event $E_{g, \varepsilon/2}$. By Chebyshev's tail bound, we have, for all sufficiently large positive integers $N$,
\begin{align}
    \bbP\!\left( E^c_{g, \varepsilon/2} \right) &\le \sum_{k \in g \mathbb{Z} \cap [N+1]} \bbP\! \left( \left| X_k - \frac{k(k-1)}{2N} \right| \ge \frac{\varepsilon N}{2} \right) \\
    &\le \sum_{k \in g \mathbb{Z} \cap [N+1]} \frac{\Var(X_k)}{(\varepsilon N/2)^2} \\
    &\le \frac{2}{\varepsilon^2 N^3} \sum_{k \in g \mathbb{Z} \cap [N+1]} k^2 &\text{(by Proposition~\ref{prop:variance-bound})} \\
    &< \frac{3}{\varepsilon^3} \cdot \frac{\log N}{N \, \log \log N}.
\end{align}
Combining this with Proposition~\ref{prop:reduction}, we finish the proof.
\end{proof}

\bigskip

\section*{Acknowledgments}
I would like to thank Morris~Ang, Alexei~Borodin, Matthew~Nicoletti, Alex~Postnikov, Sahana~Vasudevan, and Wijit~Yangjit for insightful discussions. I would like to thank Alex~Postnikov specifically for telling me about Proposition~\ref{prop:continuous-family} and showing his proof to me, which led to the discussion of waterfalls in this paper. I am grateful for Alex~Postnikov and Alexei~Borodin specifically for their encouragement. I would also like to thank Richard~Kenyon for sharing with me a copy of the slides from his ``permutons" talk. I would like to thank Sorawee~Porncharoenwase for algorithmic insights and technical help. I used {\em Polymake}, {\em R}, {\em Racket}, and {\em Wolfram~Alpha} to help with computations.

\bibliographystyle{alpha}
\bibliography{ref}

\end{document}